\documentclass[12pt,a4paper,oneside]{amsart}
\usepackage{amsfonts, amsmath, amssymb, amsthm}
\usepackage[colorlinks=true,citecolor=blue]{hyperref}
\usepackage[margin=1.4in]{geometry}
\usepackage{txfonts}
\usepackage[T1]{fontenc}
\usepackage{bbm}
\usepackage{mathtools}
\newcommand{\defeq}{\vcentcolon=}
\usepackage{graphicx}
\usepackage{enumerate}
\usepackage{verbatim}
\usepackage{tikz}
\usepackage{algorithmic}

\usetikzlibrary{decorations,decorations.pathreplacing}

\begingroup
    \makeatletter
    \@for\theoremstyle:=definition,remark,plain\do{%
        \expandafter\g@addto@macro\csname th@\theoremstyle\endcsname{%
            \addtolength\thm@preskip\parskip
            }%
        }
\endgroup

\newtheorem{theorem}{Theorem}%[section]
\newtheorem*{theorem*}{Theorem}
\newtheorem{lemma}[theorem]{Lemma}

\newtheorem*{claim*}{Claim}

\theoremstyle{definition}

\newtheorem*{remark*}{Remark}

\newtheorem{example}[theorem]{Example}

\begin{document} 

\title{On a transversal theorem of Montejano and Karasev}  

\author{Andreas F. Holmsen}

\date{\today}

% \address{Andreas F. Holmsen, %\hfill \hfill \linebreak 
% Department of Mathematical Sciences, % \hfill \hfill \linebreak
% KAIST, 
% Daejeon, South Korea.  \hfill \hfill }
% \email{andreash@kaist.edu}

\begin{abstract} 
We give a new proof a theorem of Montejano and Karasev \cite{roman} regarding $k$-dimensional transversals to small families of convex sets. While the result in \cite{roman} uses technical algebraic and topological tools, our proof is a simple application of the Borsuk--Ulam theorem. Additionally, in certain cases we obtain stronger results than those in \cite{roman}. 
\end{abstract}

%\dedicatory{My dedication}

\maketitle 

\section{Introduction}

Let $F_1, \dots, F_n$ be finite families of convex sets in $\mathbb{R}^d$.  We say that the families $F_1, \dots, F_n$ satisfy {\em the colorful intersection property} if $C_1\cap \cdots \cap C_n\neq \emptyset$ for any choice of $C_1\in F_1, \dots, C_n\in F_n$.

\smallskip

Given $n=d+1$ finite families of convex sets in $\mathbb{R}^d$ which satisfy the colorful intersection property, the colorful Helly theorem \cite{colhel} asserts that there is a point that intersects every member of one of the families $F_i$.
(Even more can be said; If $n=d+k$ there is a point that intersects every member of $k$ of the families.)

\smallskip

On the other hand, what can be said  when $n\leq d$~? In this case, the conclusion of the colorful Helly theorem clearly fails; Simply take each $F_i$ to be a family of hyperplanes in general position. Nevertheless, the colorful intersection property imposes non-trivial geometric constraints on the families. For instance, if $n=d$ it is easy to see  
%follows easily from the colorful Helly theorem 
that for any direction there exists a family $F_i$ whose members can be intersected by a {\em line} in that given direction. (Just project to a hyperplane orthogonal to the given direction and apply the colorful Helly theorem.)

\smallskip

These observations have led to a number of interesting results and conjectures concerning families satisfying the colorful intersection property \cite{edgardo, luis, roman}. Here we focus on the remarkable results by Montejano and Karasev \cite{roman}, which state that under 
certain conditions on $n$, $|F_i|$ and $d$, the colorful intersection property implies that that one of the families $F_i$ admits a {\em $k$-transversal}, that is, a $k$-dimensional affine flat that intersects every memebr of $F_i$. In its simplest non-trivial form, their theorem can be formulated as follows (See also \cite[Theorem 3.2]{luis}):

\smallskip

{\em Consider 3 red convex sets and 3 blue convex sets in $\mathbb{R}^3$, and suppose every red set intersects every blue set. Then all the red sets can be intersected by a line or all the blue sets can be intersected by a line.}

\smallskip

The general theorems of Montejano and Karasev rely on a number of technical tools (multiplication formulas for Schubert cocycles, Steifel--Whitney characteristic classes, Lusternik--Schnirelmann category of the Grassmannian), and
simpler proofs are of considerable interest. In this respect, Strausz \cite{dino} recently gave an elementary proof of the statement above based on the non-planarity of the complete bipartite graph $K_{3,3}$. 

\smallskip
The purpose of this note is to give an elementary proof of the Montejano--Karasev theorem.  (Our proof is based on the Borsuk--Ulam theorem, but it seems unlikely that this type of results can be proven without using any topology.)
In particular, we establish the following.

\begin{theorem} \label{main}
Let $k_1, \dots, k_n$ be non-negative integers and set $m = k_1 + \cdots + k_n$. For each $1\leq i \leq n$,  let $F_i$ be a family of $k_i+2$ convex sets in $\mathbb{R}^{n+m-1}$, and suppose the families $F_1, \dots, F_n$ satisfy the colorful intersection property. Then one of the families $F_i$ admits a $k_i$-transversal.
\end{theorem}

\begin{remark*} Let us point out some particular cases of Theorem \ref{main}.
\begin{enumerate}[\em (a)]
    \item The case $n=1$ is tautological, as this corresponds to a single family $F_1$ of $k+2$ convex sets in $\mathbb{R}^k$.
    \item The case $n=2$ with $k_1 = d-1$ and $k_2 = 0$ follows from the separation theorem for convex sets in $\mathbb{R}^d$. If the two members of $F_2$ are disjoint (i.e. they do not have a 0-transversal), then there is a hyperplane that strictly separates them. By the colorful intersection property every member of $F_1$ must cross this hyperplane. 
    \item The case $n\geq 2$ and $k_1 = \cdots  = k_n = 0$  corresponds to $n$ families, each consisting of two convex sets in $\mathbb{R}^{n-1}$. The colorful Helly theorem implies that the members of one of the $F_i$ must intersect.
    \item In the case $n=2$ we have $k_1+2$ red convex sets and $k_2+2$ blue convex sets in $\mathbb{R}^{k_1+k_2+1}$ such that every red set intersects every blue set. The conclusion is that the red sets admit a $k_1$-transversal or the blue sets admit a $k_2$-transversal. This is a special case of a result of Montejano and Karasev \cite[Corollary 7 with $k=1$]{roman}. (See also \cite[Theorem 3.2]{luis}.)
    \item For $n>2$ Theorem \ref{main} strengthens a result of Montejano and Karasev \cite[Theorem 8]{roman}. Their result states that if $F_1, \dots, F_n$ are families of $k+2$ convex sets in $\mathbb{R}^{n+k}$ which satisfy the colorful intersection property, then one of the families has a $k$-transversal. (See also Theorems 3.3 and 3.4 in \cite{luis}.) Our Theorem \ref{main} shows that dimension of the ambient space can be increased to $n(k+1)-1$ without affecting the conclusion.
\end{enumerate}
\end{remark*}

\section{Optimality of dimension in Theorem \ref{main}}

Before getting to the proof of Theorem \ref{main} we give an example which shows that the dimension $n+m-1$ can not be increased to $n+m$. 
Let $X$ be a set of $2n+m$ points in $\mathbb{R}^{n+m}$ and consider a partition $X = X_1 \cup \dots \cup X_n$ where $|X_i| = k_i+2$.
We assume that $X$ is in general position in the sense that any subset of size $n+m$ spans a unique affine hyperplane and the affine spans of the $X_i$ form a generic collection of flats. 
Let $\pi_i$ denote the orthogonal projection map from $\mathbb{R}^{n+m}$ to the affine span of $X_i$, and let \[F_i = \{\pi^{-1}(x) : x\in X_i\}.\] 
Since each member of $F_i$ is a flat of dimension $n+m-(k_i+1)$, if we pick one member from each $F_i$, then they will generically intersect in a unique point. In other words, the families $F_1, \dots, F_n$ satisfy the colorful intersection property. 
Moreover, $F_i$ does not admit a $k_i$-transversal. To see this, suppose $\gamma$ is a $k_i$-dimensional flat that intersects every member of $F_i$. Then $X_i$ must be contained in $\pi_i(\gamma)$, contradicting the affine independence of $X_i$.

\section{Two geometric lemmas}

The proof of Theorem \ref{main} requires two simple geometric lemmas. They are both quite standard (and most likely known in some form or another), but for completeness we include proofs.

\smallskip

The first lemma  can be thought of as an extension of Radon's lemma. It was previously observed by Goodman and Pollack \cite{eli}, and also plays a role in the argument by Strausz \cite{dino}.

\begin{lemma}\label{separation}
Let $F = \{C_1, \dots, C_{k+2}\}$ be a family of convex sets in $\mathbb{R}^d$. Then $F$ has a $k$-dimensional transversal if and only if there exists a partition $[k+2] = A \cup B$ such that 
\[\text{\em conv}\big( {\textstyle\bigcup\limits_{i\in A}} C_i\big) \cap \text{\em conv} (
{\textstyle\bigcup\limits_{j\in B}} C_j) \neq \emptyset \]
\end{lemma}

\begin{proof}
Suppose there is a $k$-dimensional transversal $\gamma$ to $F$. For each member $C_i\in F$ choose a point $p_i \in \Gamma \cap C_i$. Applying Radon's lemma to the set $\{p_1, \dots, p_{k+2}\}$ contained in the $k$-dimensional affine subspace $\gamma$ gives us the desired partition. 

For the opposite direction suppose there is a partition $[k+2] = A\cup B$ and a point $p\in \text{conv}\big( {\textstyle\bigcup_{i\in A}} C_i\big) \cap \text{conv} (
{\textstyle\bigcup_{j\in B}} C_j)$. For each $i\in A$ we can choose a point $p_i\in C_i$ such that $p\in \text{conv } \{p_i\}_{i\in A}$. Similarly, we can choose points $p_j\in C_j$ such that $p\in \text{conv } \{p_j\}_{j\in B}$.  Clearly the affine span of $\{p_i\}_{i\in A\cup B}$ intersects every member of $F$ and has dimension at most $|A|+|B|-2 = k$.
\end{proof}

The second lemma is a simple application of the separation theorem for convex sets.

\begin{lemma}\label{acyclic}
Let $v_1, \dots, v_n \in \mathbb{R}^d$ and $(b_1, \dots, b_n)\in \mathbb{R}^n$. Suppose there exists vectors $u, w\in \mathbb{R}^d$ such that 
\[w \cdot v_i < b_i < u \cdot v_i \]
for every $1\leq i \leq n$. Then there exists a vector $v\in \mathbb{R}^d$ such that $v\cdot v_i > 0$ for every $1\leq i\leq n$.
\end{lemma}

\begin{proof}
If the conclusion of the Lemma does not hold then the origin is contained in the convex hull of the $v_i$, so there 
exists a linear dependency $\sum_{i\in I}t_iv_i = 0$ with $t_i>0$ for some $\emptyset \neq I\subset [n]$. This is a contradiction since
\[0 =
w\cdot \big( \textstyle{\sum\limits_{i\in I}}t_iv_i \big) = 
\textstyle{\sum\limits_{i\in I}} t_i(w\cdot v_i) <
\textstyle{\sum\limits_{i\in I}} t_ib_i < \textstyle{\sum\limits_{i\in I}} t_i(u\cdot v_i) = u\cdot \big( \textstyle{\sum\limits_{i\in I}}t_iv_i \big) = 0. \qedhere\]
\end{proof}

\section{Proof of Theorem \ref{main}}

%\begin{proof}[Proof of Theorem \ref{main}]
Set $d = n+m-1$, and suppose that none of the subfamilies $F_i$ admits a $k_i$-transversal. 
Define $K_i$  to be the abstract
simplicial complex whose vertex set $V(K_i)$ consists of the nonempty proper subfamilies of $F_i$ and whose faces are the chains ordered by inclusion. In other words, \[K_i \defeq \{\sigma_1\subset \sigma_2 \subset \cdots \subset \sigma_r : \emptyset \neq \sigma_j \subsetneq F_i \}.\]

Note that $K_i$ is isomorphic to the barycentric subdivision of the $k_i$-skeleton of the $(k_i+1)$-dimensional simplex, and is therefore homeomorphic to  $S^{k_i}$. 
Furthermore, we observe that $\upsilon_i(\sigma) \defeq F_i\setminus \sigma$ defines a free simplicial involution on $K_i$ since taking complements reverses inclusions
\[ \sigma_1 \subset \sigma_2 \subset \cdots \subset \sigma_r  \xmapsto{\: \; \upsilon_i \; \: } \upsilon_i({\sigma}_r) \subset \cdots \subset \upsilon_i({\sigma}_2) \subset \upsilon({\sigma}_1). \]

If we group the vertices of $K_i$ into complementary pairs $\{\sigma, \upsilon_i({\sigma})\}$, then Lemma \ref{separation} implies that
for each complementary pair there exists an affine hyperplane $h_{\{\sigma, \upsilon_i({\sigma})\}}$ in $\mathbb{R}^{d}$ which strictly separates the members of $\sigma$ from the members of $\upsilon_i({\sigma})$. Let $H_{\sigma}$ and $H_{\upsilon_i(\sigma)}$ denote the opposite open halfspaces bounded by $h_{\{\sigma, \upsilon_i({\sigma})\}}$ which contain the members of $\sigma$ and $\upsilon_i(\sigma)$, respectively, and let $f_i(\sigma)$ and $f_i(\upsilon_i(\sigma))$ be their outward unit normal vectors.  
In this way, we obtain a mapping 
\[ f_i :  V(K_i) \to S^{d-1},\]
which satisfies
$f_i(\sigma) = - f_i(\upsilon_i(\sigma)) $.

\smallskip

Let $K$ be the join $K = K_1 * K_2 * \cdots * K_n$ which comes equipped with the free simplicial involution $\upsilon = \upsilon_1 * \upsilon_2 * \cdots * \upsilon_n$. Note that $K$ is homeomorphic to $S^{d}$. We define a map $f : V(K) \to S^{d-1}$ given by
\[f(\sigma) = f_i(\sigma) \iff \sigma \in V(K_i),\] which obviously  satisfies $f(\sigma)  = -f(\upsilon(\sigma))$. By taking the affine extension of $f$ we get a continuous equivariant map 
$\hat{f} : K \to \mathbb{R}^d$ for which the following holds. 
\begin{claim*}
For any simplex $S\in K$, the origin is not contained in $\hat{f}(S)$
\end{claim*}
Assuming this claim is true, the proof the theorem is complete since $\hat{f}$ would contradict the Borsuk--Ulam theorem.  

\smallskip

It remains to prove the claim. Consider a maximal simplex $S\in K$.  
The vertices of $S$ can be expressed as $n$ chains 
\[\arraycolsep=1.5pt\def\arraystretch{1.7}
\begin{array}{ccccccc}
     \sigma^{(1)}_1 & \subset & \sigma^{(1)}_2 & \subset & \cdots & \subset & \sigma^{(1)}_{k_1+1} \\
     \sigma^{(2)}_1 & \subset & \sigma^{(2)}_2 & \subset & \cdots & \subset & \sigma^{(2)}_{k_2+1} \\
     &&& \vdots &&&\\
     \sigma^{(n)}_1 & \subset & \sigma^{(n)}_2 & \subset & \cdots & \subset & \sigma^{(n)}_{k_n+1}
\end{array}\]
where $\sigma^{(i)}_j$ is a subfamily of $F_i$ with $|\sigma^{(i)}_j| = j$. For every $1\leq i \leq n$ let us denote $\sigma^{(i)}_1 = \{C_i\}$ and $\nu_i(\sigma^{(i)}_{k_i+1}) = \{D_i\}$. We observe that \[C_i \subset {\textstyle \bigcap\limits_{j=1}^{k_i+1}} H_{\sigma^{(i)}_j}
\;\;\; \text{ and } \;\;\;  
D_i \subset {\textstyle \bigcap\limits_{j=1}^{k_i+1}} H_{\upsilon_i(\sigma^{(i)}_j)}\]
so by the colorful intersection property, it follows that the intersections $\bigcap_{i=1}^n C_i$ and $\bigcap_{i=1}^n D_i$ are both nonempty. This implies 

\[{\textstyle \bigcap\limits_{\sigma\in S}} H_{\sigma} \neq \emptyset 
\;\;\; \text{and} \;\;\; 
{\textstyle \bigcap\limits_{\sigma\in S}} H_{\upsilon(\sigma)} \neq \emptyset.
\]

Since $H_{\sigma}$ and $H_{\upsilon(\sigma)}$ are opposite open halfspaces bounded by $h_{\{\sigma,\upsilon(\sigma)\}}$, it follows from Lemma \ref{acyclic} that  $0\notin \text{conv}\hspace{0.03cm} f(S) = \hat{f}(S)$. This completes the proof of the claim and of Theorem \ref{main}.
%\end{proof}

\section{Final remarks}

The paper by Montejano and Karasev \cite{roman} contains several other interesting results that are not covered by our Theorem \ref{main}. In particular, Corollary 7 in \cite{roman} gives conditions for when families of size greater than $k+2$ have a $k$-transversal. For example they show the following:

\medskip

{\em Let $F_1, F_2, F_3$ be families of convex sets in $\mathbb{R}^5$ satisfying the colorful intersection property, and suppose $|F_1| = |F_2| = |F_3| =  5$. Then one of the families $F_i$ admits a 2-transversal.}

\medskip

We have not been able to prove the statement above using the proof method of Theorem \ref{main}. The main issue seems to be that for families of size greater than $k+2$ we do not have a good combinatorial characterization for the existence of a $k$-transversal as in Lemma \ref{separation}. 

\smallskip

{\em Question.} What is the optimal dimension in Corollary 7 in \cite{roman} ?


\begin{thebibliography}{99}

\bibitem{colhel}  I.~B{\'a}r{\'a}ny, A generalization of Carath{\'e}odory's theorem,  Discrete Math. 40, 141--152 (1982)

\bibitem{eli} J.~E.~Goodman and R.~Pollack, Hadwiger's transversal theorem in higher dimensions, J. Amer. Math. Soc. 1, 301--309 (1988)

\bibitem{edgardo} L.~Mart{\'i}nez--Sandoval, E.~Rold{\'a}n--Pensado and N.~Rubin, Further Consequences of the Colorful Helly Hypothesis, Discrete Comput. Geom. 63,  848--866 (2020)

\bibitem{luis} L.~Montejano, Transversals, topology and colorful geometric results, In {\em Geometry -- Intuitive, discrete, and convex}, Bolyai Society Mathematical Studies 24, 205--218 (2013)

\bibitem{roman} L.~Montejano and R.~Karasev, Topological transversals to a family of convex sets, Discrete~Comput.~Geom. 46, 283--300 (2011).

\bibitem{dino} R.~Strausz, How do 9 points look like in $\mathbb{E}^3$?, Ann. Comb. {\em (to appear)}

%\bibitem{wenger} R.~Pollack and R.~Wenger, Necessary and sufficient conditions for hyperplane transversals, Combinatorica 10, 307--311 (1990).

\end{thebibliography}
\end{document}